\documentclass{amsart}
\usepackage{amsfonts}

\newtheorem{theorem}{Theorem}
\theoremstyle{plain}

\newtheorem{case}{Case}

\newtheorem{corollary}{Corollary}

\newtheorem{definition}{Definition}

\numberwithin{equation}{section}

\input{tcilatex}

\begin{document}
\title[LC HELICES AND HARMONIC CURVATURES IN SPACE FORMS (HYPERSURFACE)]{LC
HELICES AND HARMONIC CURVATURES IN SPACE FORMS (HYPERSURFACE)}
\author{Ali \c{S}ENOL$^{(1)}$}
\address{Cankiri Karatekin University, Faculty of Science, Department of
Mathematics, 18100 Cankiri, Turkey}
\email{asenol@karatekin.edu.tr}
\author{Evren ZIPLAR$^{(2)}$}
\address{Ankara University, Faculty of Science, Department of Mathematics,
06100 Ankara, Turkey}
\email{evrenziplar@yahoo.com}
\author{Yusuf YAYLI$^{(3)}$}
\address{Ankara University, Faculty of Science, Department of Mathematics,
06100 Ankara, Turkey}
\email{yyayli@science.ankara.edu.tr}

\begin{abstract}
In $n$-dimensional Euclidean space $E^{n}$, harmonic curvatures of a
non-degenerate curve defined by \"{O}zdamar and Hacisaliho\u{g}lu $\left[ 4%
\right] $. In this paper, We define a new type of curves called LC helix
when the angle between tangent of this curve and LC parallel vector field in
space form is constant. Furthermore, several characterizations of these
curves by using its harmonic curvatures are obtained. Particularly, in the
3-dimensional spaceform we obtain the results $\left[ 5\right] .$
\end{abstract}

\maketitle

\section{\textbf{Introduction}}

Helices in $E^{3}$ are curves whose tangents make a constant angle with a
fixed straight line. In $1802$, Lancret proved that the necessary and
sufficient condition for a curve to be a helix is that the ratio of its
curvature be constant $[1]$. In fact, circular helix is the simplest
three-dimensional spirals. In this paper, we give some characterizations for
a non-degenerate curve $\alpha \,$to be a generalized helix by using
harmonic curvatures of the curve in $n-$dimensional Euclidean space $E^{n}.$%
\thinspace Also, we obtain a vector $D$\thinspace \thinspace for a
non-degenerate curve $\alpha \,$and we called it $a\,generalized$ $%
Darboux\,vector$. In this study, we gave the LC helix definition for the
hypersurfaces. For these helices, we found some new theorems. If we get $%
E^{n}$ as the hypersurfaces, LC helices and the generalized helices coincide
and the results $\left[ 2\right] $ are obtained.

For a$\,\,X$ \thinspace \thinspace vector field along the parameter curve of 
$\alpha :I\rightarrow E^{n}\,\,$in $E^{n},$ if%
\begin{equation*}
\overset{.}{X}=\frac{dX}{dt}=0
\end{equation*}%
then $X$ vector field is called paralell in terms of Euclidean along curve $%
\alpha .$Let $M$ be a hypersurface on $E^{n}$\thinspace and $\alpha
:I\rightarrow M\,\,\,$a paremeter curve on $M$.

if $\,\,\nabla _{v_{1}}^{X}=0$ for a differentiable $X$ \thinspace tangent
vector field to the surface $M$ along the curve $\alpha ,$ then $X$ is a
unit parallel vector field of Levi-civita meanning ($\nabla \,\,$is the
Riemann connection of the surface $M$).

\section{\textbf{Preliminaries}}

\begin{definition}
Let $M(c)\,$be space form (or Hypersurface) that has a sectional curvature $%
c\,$and $\alpha \,$is a curve in $M(c)\,\ v_{1}~$is a tangent vector field
of $\alpha $ and $X$ is a unit parallel vector field of Levi-civita meanning
as defined above 
\begin{equation*}
\left\langle v_{1},X\right\rangle =sbt=\cos \varphi
\end{equation*}%
we call this curve as $LC\,\,helix\,\,X$ is axis of a $LC\,\,helix\,$of $%
\alpha .\,$In order to characterize a curve (as \textit{inclined curves}) in 
$M\left( c\right) $ space form.
\end{definition}

We will generalize the harmonic curvature known as, 
\begin{equation*}
H_{1}=\frac{\kappa }{\tau }
\end{equation*}%
for $n=3$, to a higher order harmonic curvature for a curve in $M\left(
c\right) $ space.

\begin{definition}
\bigskip Let be all curvatures $k_{i}$ $(i=1,2,...,n-1)$ of the curve in $%
I\subset R,$ $v_{1}$ is a tangent vector field of $\alpha $ in $E^{n}$%
.Harmonic curvatures of $\alpha $ are defined by 
\begin{equation*}
H_{i}:I\rightarrow R
\end{equation*}%
\begin{equation*}
H_{i}=\left\{ 
\begin{array}{c}
\frac{\kappa _{1}}{\kappa _{2}},\,\,\,\,\,\,\,\,\,\,\,\,\,\,\,\,\,\,\,\,\,\,%
\,\,\,\,\,\,\,\,\,\,\,\,\,\,\,\,\,\,\,\,\,\,\,i=1 \\ 
\left\{ \nabla _{v_{1}}^{H_{i-1}}+H_{i-2}\kappa _{i}\right\} \frac{1}{\kappa
_{i+1}},...1\langle i\leq n-2%
\end{array}%
\right\} .
\end{equation*}
\end{definition}

Specifically, this definition for the hypersurface in $E^{n}~$coincide with
the definition of harmonic curvatures in $\left[ 2\right] .$

\section{\textbf{Harmonic curvatures and LC helices}}

In this section, we give some characterizations for LC helices by using the
harmonic curvatures of the curve.

\begin{theorem}
\bigskip Let $M(c)\,$be space form (or Hypersurface) that has a sectional
curvature $c.~$Let $\left\{ v_{1},v_{2},...v_{n}\right\} \,,\left\{
H_{1},H_{2},....H_{n-2}\right\} $ be denote the Frenet frame and the higher
ordered harmonic curvatures of the curve, respectively. Then the following
equations is holds where $X$ is axis of a LC helix $\alpha $, 
\begin{equation*}
\left\langle v_{i+2},X\right\rangle =H_{i}.\left\langle v_{1},X\right\rangle
,\,\,\,\,\,\,\,1\,\,<i\leq n-2\,\,.\,\,\,\,\,\,
\end{equation*}
\end{theorem}

\begin{proof}
\begin{equation*}
X=\lambda _{1}v_{1}+\lambda _{2}v_{2}+...+\lambda _{n}v_{n}
\end{equation*}%
\begin{equation*}
\left\langle X,v_{1}\right\rangle =\lambda _{1}=\cos \theta
\end{equation*}%
\begin{equation*}
\left\langle \nabla _{v_{1}}X,v_{1}\right\rangle +\left\langle X,\nabla
_{v_{1}}^{v_{1}}\right\rangle =0
\end{equation*}%
since $\nabla _{v_{1}}^{X}=0$, 
\begin{equation*}
\left\langle X,\kappa v_{2}\right\rangle =0
\end{equation*}%
then $\kappa \neq 0$ we obtain 
\begin{equation*}
\left\langle X,v_{2}\right\rangle =0
\end{equation*}%
take the derivative again 
\begin{equation*}
\left\langle \nabla _{v_{1}}^{X},v_{2}\right\rangle +\left\langle X,\nabla
_{v_{1}}^{v_{2}}\right\rangle =0
\end{equation*}%
since $\nabla _{v_{1}}^{X}=0$ , 
\begin{equation*}
\left\langle X,-\kappa _{1}v_{1}+\kappa _{2}v_{3}\right\rangle =0
\end{equation*}%
and if the following is re-arranged 
\begin{equation*}
-\kappa _{1}\left\langle X,v_{1}\right\rangle +\kappa _{2}\left\langle
X,v_{3}\right\rangle =0
\end{equation*}%
we get 
\begin{equation*}
\lambda _{3}=\frac{\kappa _{1}}{\kappa _{2}}.\cos \theta
\end{equation*}%
\begin{equation*}
H_{1}=\frac{\kappa _{1}}{\kappa _{2}}
\end{equation*}%
\begin{equation*}
\left\langle v_{3},X\right\rangle =H_{1}\cos \theta
\end{equation*}%
\begin{equation*}
\left\langle v_{3},X\right\rangle =H_{1}\left\langle v_{1},X\right\rangle
\end{equation*}%
for i=1 the proof is verified. if \thinspace we prove for induction i. Since
the theorem holds for i-1, let us prove it by induction for i%
\begin{equation*}
\left\langle v_{i+1},X\right\rangle =H_{i-1}\left\langle v_{1},X\right\rangle
\end{equation*}%
if we take the derivative, 
\begin{equation*}
\left\langle \nabla _{v_{1}}^{v_{i+1}},X\right\rangle +\left\langle \nabla
_{v_{1}}^{X},v_{i+1}\right\rangle =\nabla _{v_{1}}^{H_{i-1}}.\left\langle
v_{1},X\right\rangle
\end{equation*}%
\begin{equation*}
\left\langle -\kappa _{i}v_{i}+\kappa _{i+1}v_{i+2},X\right\rangle =\nabla
_{v_{1}}^{H_{i-1}}.\left\langle v_{1},X\right\rangle
\end{equation*}%
and re-arrange the expression, 
\begin{equation*}
-\kappa _{i}\left\langle v_{i},X\right\rangle +\kappa _{i+1}\left\langle
v_{i+2},X\right\rangle =\nabla _{v_{1}}^{H_{i-1}}.\left\langle
v_{1},X\right\rangle
\end{equation*}%
\begin{equation*}
\left\langle v_{i+2},X\right\rangle =\left\{ \kappa _{i}\left\langle
v_{i},X\right\rangle +\nabla _{v_{1}}^{H_{i-1}}.\left\langle
v_{1},X\right\rangle \frac{1}{\kappa _{i+1}}\right\} .
\end{equation*}%
\begin{equation*}
\left\langle v_{i},X\right\rangle =H_{i-2}\left\langle v_{1},X\right\rangle
\end{equation*}%
substituded, 
\begin{equation*}
\left\langle v_{i+2},X\right\rangle =\left\{ \kappa _{i}H_{i-2}+\nabla
_{v_{1}}^{H_{i-1}}.\right\} \frac{1}{\kappa _{i+1}}\left\langle
v_{1},X\right\rangle
\end{equation*}%
\begin{equation*}
\left\langle v_{i+2},X\right\rangle =H_{i}\left\langle v_{1},X\right\rangle
\end{equation*}%
Therefore the proof is verified via induction. if $M=E^{n}\subset E^{n+1}\,$%
\thinspace for $c=0\,\ \,$then. we obtain $\left[ 2\right] .$
\end{proof}

\begin{corollary}
If $X$ is axis of a LC helix $\alpha $, then we can write, 
\begin{equation*}
X=\lambda _{1}v_{1}+\lambda _{2}v_{2}+...+\lambda _{n}v_{n}\text{ }.
\end{equation*}%
for the Theorem 3, we get%
\begin{equation*}
\lambda _{i}=\left\langle X,v_{i}\right\rangle =H_{i-2}\left\langle
v_{1},X\right\rangle
\end{equation*}%
where 
\begin{equation*}
\left\langle v_{1},X\right\rangle =sbt\neq 0=\cos \varphi
\end{equation*}%
By the definition of harmonic curvature, we obtain 
\begin{equation*}
X=\cos \varphi (v_{1}+H_{1}v_{3}+...+H_{n-2}v_{n})\text{ }.
\end{equation*}%
Also, 
\begin{equation*}
D=v_{1}+H_{1}v_{3}+...+H_{n-2}v_{n}
\end{equation*}%
is axis of LC helix $\alpha .$ We find 
\begin{equation*}
X=\frac{D}{\left\Vert D\right\Vert }=\frac{v_{1}+H_{1}v_{3}+...+H_{n-2}v_{n}%
}{\left\Vert v_{1}+H_{1}v_{3}+...+H_{n-2}v_{n}\right\Vert }\text{ }.
\end{equation*}%
Here, we can easily pove that \ \textquotedblleft \thinspace $X$ \thinspace
\thinspace is a LC parallel vector field\textquotedblright $\Leftrightarrow $%
\textquotedblleft \thinspace $D\,$\thinspace \thinspace \thinspace is a LC
parallel vector field\textquotedblright .
\end{corollary}

\begin{definition}
Let $M(c)\,$be space form (or Hypersurface) that has a sectional curvature $%
c.~$Let $\left\{ v_{1},v_{2},...v_{n}\right\} \,,\left\{
H_{1},H_{2},....H_{n-2}\right\} $ be denote the Frenet frame and the higher
ordered harmonic curvatures of the curve, respectively. The vector $D\,$%
defined by $\,D=v_{1}+H_{1}v_{3}+...+H_{n-2}v_{n}$ is called the generalized
Darboux vector of the curve $\alpha .$
\end{definition}

\begin{theorem}
Let $M(c)\,$be space form (or Hypersurface) that has a sectional curvature $%
c.$ Let $\left\{ v_{1},v_{2},...v_{n}\right\} \,,\left\{
H_{1},H_{2},....H_{n-2}\right\} $ be denote the Frenet frame and the higher
ordered harmonic curvatures of the curve, respectively. Then,%
\begin{equation*}
\text{\textquotedblleft }\,\,\alpha ~\text{is }LC\text{ helix%
\textquotedblright }\,\,\Leftrightarrow \,\,\text{\textquotedblleft }%
\,\,D\,\ \,\text{is }LC~\ \text{parallel~vector~field\textquotedblright .}
\end{equation*}
\end{theorem}

\begin{proof}
$\left( \Leftarrow \right) \,\,$if$\,\,D$ is parallel vector field then we
can $\left\Vert D\right\Vert =$constant%
\begin{equation*}
X=\frac{D}{\left\Vert D\right\Vert }=\frac{1}{\left\Vert D\right\Vert }D
\end{equation*}%
LC becomes a parallel vector field. 
\begin{equation*}
\left\langle X,v_{1}\right\rangle =\frac{1}{\left\Vert D\right\Vert }=\text{%
constant}\,\,\,\,\,\,,\,\,\,\,\,\,\left\Vert X\right\Vert =1
\end{equation*}%
$\alpha $ is a LC helix.

$\left( \Rightarrow \right) \,$Let $\alpha $ be a LC helix. $\alpha \,$%
\thinspace axis defined as 
\begin{equation*}
X=\lambda _{1}v_{1}+\lambda _{2}v_{2}+...+\lambda _{n}v_{n}\text{ .}
\end{equation*}%
\newline
For $\lambda _{1}=\left\langle v_{1},X\right\rangle =$constant$,$ there is a 
$X\,\,\,$\thinspace LC vector field where$\,\nabla _{v_{1}}^{X}=0$. 
\begin{equation*}
\left\langle \nabla _{v_{1}}^{v_{1}},X\right\rangle +\left\langle \nabla
_{v_{1}}^{X},v_{1}\right\rangle =0
\end{equation*}%
\begin{equation*}
\left\langle \kappa _{1}v_{2},X\right\rangle =0
\end{equation*}%
\begin{equation*}
\lambda _{2}=\left\langle v_{2},X\right\rangle =0
\end{equation*}%
\begin{equation*}
\left\langle X,v_{2}\right\rangle =0\text{ }.
\end{equation*}%
Taking the derivative, 
\begin{equation*}
\left\langle \nabla _{v_{1}}^{X},v_{2}\right\rangle +\left\langle X,\nabla
_{v_{1}}^{v_{2}}\right\rangle =0
\end{equation*}%
\begin{equation*}
\left\langle X,-\kappa _{1}v_{1}+\kappa _{2}v_{3}\right\rangle =0
\end{equation*}%
\begin{equation*}
\lambda _{3}=\left\langle v_{3},X\right\rangle =H_{1}\left\langle
v_{1},X\right\rangle =H_{1}\cos \varphi
\end{equation*}%
\begin{equation*}
\lambda _{i}=\left\langle v_{i},X\right\rangle =H_{i-2}\left\langle
v_{1},X\right\rangle
\end{equation*}%
and$\,\,$%
\begin{equation*}
X=\cos \varphi (v_{1}+H_{1}v_{3}+...+H_{n-2}v_{n})
\end{equation*}%
or 
\begin{equation*}
X=\cos \varphi .D\text{ .}
\end{equation*}%
Since $X$ \thinspace is \thinspace a LC vector field, $\nabla _{v_{1}}^{X}=0$%
\thinspace \thinspace is obtained and since $\cos \varphi =$ constant%
\begin{eqnarray*}
\nabla _{v_{1}}^{X} &=&\cos \varphi .\nabla _{v_{1}}^{D} \\
\nabla _{v_{1}}^{D} &=&0
\end{eqnarray*}%
$D\,$\thinspace is $LC$ parallel vector field.
\end{proof}

\begin{theorem}
Let $\alpha (s)$ be a unit speed curve in $M\left( c\right) $ space form
(hypersurface) with Frenet vectors\ $\left\{ v_{1},v_{2},...v_{n}\right\} $%
and harmonic curvatures $\left\{ H_{1},H_{2},....H_{n-2}\right\} $ \ then, 
\begin{equation*}
^{\prime \prime }\alpha \text{ is a }LC\text{ helix}"\text{ }\Rightarrow
\,\sum\limits_{i=1}^{n-2}H_{i}^{2}=\text{constant.}
\end{equation*}
\end{theorem}

\begin{proof}
\textbf{\thinspace }Let $\alpha $ be a LC helix in $M$. In this case, 
\begin{equation*}
\left\langle v_{1},X\right\rangle =\cos \varphi =cons\tan t
\end{equation*}%
According to Theorem 3%
\begin{equation*}
\left\langle v_{i+2},X\right\rangle =H_{i}\left\langle v_{1},X\right\rangle
\,.
\end{equation*}%
Also, 
\begin{equation*}
\left\langle v_{1},X\right\rangle =\text{constant}
\end{equation*}%
\begin{equation*}
\left\langle \nabla _{v_{1}}^{v_{1}},X\right\rangle +\left\langle \nabla
_{v_{1}}^{X},v_{1}\right\rangle =0
\end{equation*}%
\begin{equation*}
\left\langle \kappa v_{2},X\right\rangle =0
\end{equation*}%
\begin{equation*}
\left\langle v_{2},X\right\rangle =0
\end{equation*}%
Since $\left\{ v_{1},v_{2},...v_{n}\right\} $ ortonormal system will be an
ortonormal base of $\chi (M)$.%
\begin{equation*}
X\in S_{p}\left\{ v_{1},v_{2},...v_{n}\right\} 
\end{equation*}%
or%
\begin{equation*}
X=\overset{n}{\sum_{i=1}}\left\langle v_{i},X\right\rangle v_{i}
\end{equation*}%
can be written. If $\left\langle v_{1},X\right\rangle =\cos \varphi
,\,\left\langle v_{2},X\right\rangle =0,$ and $\left\langle
v_{3},X\right\rangle =H_{1}\cos \varphi \,$\thinspace are substitued for $%
\left\langle v_{i},X\right\rangle $ values 
\begin{equation*}
X=\cos \varphi .v_{1}+\sum_{j=1}^{n-2}H_{j}.\cos \varphi .v_{j+2}
\end{equation*}%
is obtained. Since $X\in X(M)$ is a unit parallel vector field of
Levi-civita meanning and $\left\Vert X\right\Vert =1$, \thinspace then 
\begin{equation*}
\cos ^{2}\varphi .+\sum_{j=1}^{n-2}H_{j}^{2}.\cos ^{2}\varphi =1
\end{equation*}%
\begin{equation*}
\sum_{j=1}^{n-2}H_{j}^{2}=tg^{2}\varphi =cons\tan t
\end{equation*}%
\begin{equation*}
\sum_{j=1}^{n-2}H_{j}^{2}=cons\tan t\text{ .}
\end{equation*}%
If $M=E^{n}\subset E^{n+1}\,$\thinspace for $c=0\,\,$then we obtain $\left[ 2%
\right] .$
\end{proof}

\begin{theorem}
Let $M(c\,)\,$be a space form where $c$ is the sectional curve, and curve $%
\alpha :I\rightarrow M(c)$ be a n-th order regular curve. In this case, 
\begin{equation*}
\alpha \text{ curve is a LC helix }\Leftrightarrow V_{1}\left[ H_{n-2}\right]
+k_{n-1}H_{n-3}=0
\end{equation*}
\end{theorem}

\begin{proof}
\begin{equation*}
D=v_{1}+H_{1}v_{3}+...+H_{n-2}v_{n}
\end{equation*}%
\begin{equation*}
V_{1}\left[ H_{i-1}\right] =-k_{i}H_{i-2}+k_{i+1}H_{i}
\end{equation*}%
yields 
\begin{eqnarray*}
\text{for\thinspace }~i &=&1,\,\,\nabla _{v_{1}}^{D}=\text{ }-k_{1}v_{1} \\
\text{ for~~}i &=&2\,\,\,\text{ }+\left( -k_{2}H_{0}+k_{3}H_{2}\right)
v_{3}+H_{1}\left( -k_{2}v_{2}+k_{3}v_{4}\right) \\
\text{for~}i &=&3\,\,\,\text{ }+\left( -k_{3}H_{1}+k_{4}H_{3}\right)
v_{4}+H_{2}\left( -k_{3}v_{3}+k_{4}v_{5}\right) \\
\text{for }i &=&4\,\,\text{ }+\left( -k_{4}H_{2}+k_{5}H_{4}\right)
v_{5}+H_{3}\left( -k_{4}v_{4}+k_{5}v_{6}\right) \\
&&\vdots \\
\text{for }i &=&n-3\,\,\,\text{ }+\left(
-k_{n-3}H_{n-5}+k_{n-2}H_{n-3}\right) v_{n-2}+H_{n-4}\left(
-k_{n-3}v_{n-3}+k_{n-2}v_{n-1}\right) \\
\text{for }i &=&n-2\,\,\,+\left( -k_{n-2}H_{n-4}+k_{n-1}H_{n-2}\right)
v_{n-1}+H_{n-3}\left( -k_{n-2}v_{n-2}+k_{n-1}v_{n}\right) \\
&&+v_{1}\left[ H_{n-2}\right] v_{n}-H_{n-2}k_{n-1}v_{n-1} \\
\nabla _{v_{1}}^{D} &=&\left( v_{1}\left[ H_{n-2}\right] +k_{n-1}H_{n-3}%
\right) v_{n}
\end{eqnarray*}%
yields, 
\begin{equation*}
``\text{ }D\,\,\,\text{is }\,LC\,\,\text{parallel vector field}"\text{ if
and only if }V_{1}\left[ H_{n-2}\right] +k_{n-1}H_{n-3}=0\text{ .}
\end{equation*}%
Since 
\begin{equation*}
\alpha \text{ curve is a LC helix .if and only if \ }D\,\,\,\text{is }%
\,LC\,\,\text{parallel vector field \ if and only if \ }V_{1}\left[ H_{n-2}%
\right] +k_{n-1}H_{n-3}=0
\end{equation*}%
proof is completed.

if $M=E^{n}\subset E^{n+1}\,$\thinspace for $c=0\,\,$then we obtain $\left[ 2%
\right] .$
\end{proof}

\begin{corollary}
For $n=2m+1~\left( \text{n is odd}\right) $ lets review 
\begin{equation*}
v_{1}\left[ H_{n-2}\right] +k_{n-1}H_{n-3}=0
\end{equation*}%
if $n=2m+1$ we can consider $k_{1},k_{2},.......k_{2m-1},k_{2m}$ curvatures.
So we have the rates 
\begin{equation*}
\frac{k_{1}}{k_{2}},\frac{k_{3}}{k_{4}},......\frac{k_{2m-3}}{k_{2m-2}},%
\frac{k_{2m-1}}{k_{2m}}
\end{equation*}%
using the definition for harmonic curvature, 
\begin{equation*}
H_{2m-2}=\frac{1}{k_{2m-1}}v_{1}\left[ H_{2m-3}\right] +\frac{k_{2m-2}}{%
k_{2m-1}}H_{2m-4}
\end{equation*}%
if we multiply the following by $k_{2m}\,$and reorganize accordingly 
\begin{eqnarray*}
k_{2m}.H_{2m-2} &=&\frac{k_{2m}}{k_{2m-1}}v_{1}\left[ H_{2m-3}\right]
+k_{2m}.\frac{k_{2m-2}}{k_{2m-1}}H_{2m-4}, \\
H_{2m-4} &=&\frac{1}{k_{2m-3}}v_{1}\left[ H_{2m-5}\right] +\frac{k_{2m-4}}{%
k_{2m-3}}H_{2m-6}
\end{eqnarray*}%
Multiply the following $k_{2m}\frac{k_{2m-2}}{k_{2m-1}}\,$, 
\begin{eqnarray*}
k_{2m}.\frac{k_{2m-2}}{k_{2m-1}}.H_{2m-4} &=&\frac{k_{2m}}{k_{2m-3}}.\frac{%
k_{2m-2}}{k_{2m-1}}v_{1}\left[ H_{2m-5}\right] +k_{2m}.\frac{k_{2m-2}}{%
k_{2m-1}}\frac{k_{2m-4}}{k_{2m-3}}H_{2m-6} \\
H_{2m-6} &=&\frac{1}{k_{2m-5}}v_{1}\left[ H_{2m-7}\right] +\frac{k_{2m-6}}{%
k_{2m-5}}H_{2m-8}\text{ .}
\end{eqnarray*}%
Multiply the following $k_{2m-4}\frac{k_{2m}}{k_{2m-1}}.\frac{k_{2m-2}}{%
k_{2m-3}}$, 
\begin{eqnarray*}
k_{2m-4}.\frac{k_{2m}}{k_{2m-1}}.\frac{k_{2m-2}}{k_{2m-3}}H_{2m-6} &=&\frac{%
k_{2m}}{k_{2m-1}}.\frac{k_{2m-2}}{k_{2m-3}}.\frac{k_{2m-4}}{k_{2m-5}}v_{1}%
\left[ H_{2m-7}\right]  \\
&&+k_{2m-4}.\frac{k_{2m}}{k_{2m-1}}.\frac{k_{2m-2}}{k_{2m-3}}\frac{k_{2m-6}}{%
k_{2m-5}}H_{2m-8} \\
&&\vdots  \\
H_{6} &=&\frac{1}{k_{7}}v_{1}\left[ H_{5}\right] +\frac{k_{6}}{k_{7}}H_{4} \\
H_{4} &=&\frac{1}{k_{5}}v_{1}\left[ H_{3}\right] +\frac{k_{4}}{k_{5}}H_{2}%
\text{ .}
\end{eqnarray*}%
We obtain the equalities above. If we substitude these into the equation
above. we get 
\begin{eqnarray*}
&&v_{1}\left[ H_{2m-1}\right] +\frac{k_{2m}}{k_{2m-1}}v_{1}\left[ H_{2m-3}%
\right] +\frac{k_{2m}}{k_{2m-3}}.\frac{k_{2m-2}}{k_{2m-1}}v_{1}\left[
H_{2m-5}\right] + \\
&&\frac{k_{2m}}{k_{2m-1}}.\frac{k_{2m-2}}{k_{2m-3}}.\frac{k_{2m-4}}{k_{2m-5}}%
v_{1}\left[ H_{2m-7}\right] +...+ \\
&&\frac{k_{2m}}{k_{2m-1}}.\frac{k_{2m-2}}{k_{2m-3}}.\frac{k_{2m-4}}{k_{2m-5}}%
...\frac{k_{6}}{k_{5}}v_{1}\left[ H_{5}\right] +...+ \\
&&\frac{k_{2m}}{k_{2m-1}}.\frac{k_{2m-2}}{k_{2m-3}}.\frac{k_{2m-4}}{k_{2m-5}}%
...\frac{k_{6}}{k_{5}}.\frac{k_{4}}{k_{3}}v_{1}\left[ H_{3}\right] +...+
\end{eqnarray*}%
\begin{equation*}
+\frac{k_{2m}}{k_{2m-1}}.\frac{k_{2m-2}}{k_{2m-3}}.\frac{k_{2m-4}}{k_{2m-5}}%
...\frac{k_{8}}{k_{7}}.\frac{k_{6}}{k_{5}}...H_{2}=0
\end{equation*}%
Here, $H_{2}=\frac{1}{k_{3}}.\left( \frac{k_{1}}{k_{2}}\right) ^{\prime }$
so final equation is, 
\begin{eqnarray*}
&&v_{1}\left[ H_{2m-1}\right] +\frac{k_{2m}}{k_{2m-1}}v_{1}\left[ H_{2m-3}%
\right] +\frac{k_{2m}}{k_{2m-3}}.\frac{k_{2m-2}}{k_{2m-1}}v_{1}\left[
H_{2m-5}\right]  \\
&&+\frac{k_{2m}}{k_{2m-1}}.\frac{k_{2m-2}}{k_{2m-3}}.\frac{k_{2m-4}}{k_{2m-5}%
}v_{1}\left[ H_{2m-7}\right]  \\
&&+......+\frac{k_{2m}}{k_{2m-1}}.\frac{k_{2m-2}}{k_{2m-3}}.\frac{k_{2m-4}}{%
k_{2m-5}}......\frac{k_{6}}{k_{5}}v_{1}\left[ H_{5}\right]  \\
&&%
\begin{array}{l}
+....+\dfrac{k_{2m}}{k_{2m-1}}.\dfrac{k_{2m-2}}{k_{2m-3}}.\dfrac{k_{2m-4}}{%
k_{2m-5}}......\dfrac{k_{6}}{k_{5}}.\dfrac{k_{4}}{k_{3}}\left( \dfrac{k_{1}}{%
k_{2}}\right) ^{\prime }=0%
\end{array}%
.
\end{eqnarray*}
\end{corollary}

\begin{theorem}
Let $M(c\,)\,$be a space form where $c$ is the sectional curve, and if, for
n-th order regular curve $\alpha :I\rightarrow M(c)$, $\frac{k_{1}}{k_{2}},%
\frac{k_{3}}{k_{4}},...\frac{k_{2m-3}}{k_{2m-2}},\frac{k_{2m-1}}{k_{2m}}$
\thinspace ratios are constant then 
\begin{eqnarray*}
H_{2i} &=&0 \\
H_{2i+1} &=&\frac{k_{2i+1}}{k_{2i+2}}.\frac{k_{2i-1}}{k_{2i}}.\frac{k_{2i-3}%
}{k_{2i-2}}...\frac{k_{3}}{k_{4}}.\frac{k_{1}}{k_{2}}
\end{eqnarray*}%
$i=1,2,...,m$ .
\end{theorem}

\begin{proof}
By induction 
\begin{equation*}
\text{for\thinspace \thinspace }i=1\text{,\thinspace }H_{2}=\frac{1}{k_{3}}%
.\left( \frac{k_{1}}{k_{2}}\right) ^{^{\prime }}=0
\end{equation*}%
we get 
\begin{equation*}
H_{3}=\frac{1}{k_{4}}v_{1}\left[ H_{2}\right] +\frac{k_{3}}{k_{4}}H_{1}
\end{equation*}%
\begin{equation*}
H_{3}=\frac{k_{3}}{k_{4}}.\frac{k_{1}}{k_{2}}
\end{equation*}%
varifies\thinspace \thinspace for\thinspace \thinspace $i=1$.\thinspace

if it verifies for $i=p$ in general 
\begin{equation*}
H_{2p}=0
\end{equation*}%
\begin{equation*}
H_{2p+1}=\frac{k_{2p+1}}{k_{2p+2}}.\frac{k_{2p-1}}{k_{2p}}.\frac{k_{2p-3}}{%
k_{2p-2}}...\frac{k_{3}}{k_{4}}.\frac{k_{1}}{k_{2}}
\end{equation*}%
Therefore, 
\begin{equation*}
H_{2p+2}=\frac{1}{k_{2p+3}}v_{1}\left[ H_{2p+1}\right] +\frac{k_{2p+2}}{%
k_{2p+3}}.H_{2p}=0
\end{equation*}%
and we obtain 
\begin{equation*}
H_{2p+3}=\frac{1}{k_{2p+4}}v_{1}\left[ H_{2p+2}\right] +\frac{k_{2p+3}}{%
k_{2p+4}}.H_{2p+1}
\end{equation*}%
\begin{equation*}
H_{2p+3}=\frac{k_{2p+3}}{k_{2p+4}}.\frac{k_{2p+1}}{k_{2p+2}}.\frac{k_{2p-1}}{%
k_{2p}}.\frac{k_{2p-3}}{k_{2p-2}}...\frac{k_{3}}{k_{4}}.\frac{k_{1}}{k_{2}}
\end{equation*}
\end{proof}

Therefore it is proved that it varifies for $i=p+1$.

\begin{case}
As a result of Theorem 10 if%
\begin{equation*}
\frac{k_{1}}{k_{2}},\frac{k_{3}}{k_{4}},...,\frac{k_{2m-3}}{k_{2m-2}},\frac{%
k_{2m-1}}{k_{2m}}
\end{equation*}%
ratio are constant we can say that $\alpha \,\,\,$is a LC helix in the sense
of Hayden $[5]$.
\end{case}

\begin{case}
In case of $\dfrac{k_{1}}{k_{2}},\dfrac{k_{3}}{k_{4}},...,\dfrac{k_{2m-3}}{%
k_{2m-2}},\dfrac{k_{2m-1}}{k_{2m}}$\thinspace \thinspace ratios being
constant, we can find a $D$ vector which makes the same angle with all
Frenet vectors. $D$ vector as $%
D=v_{1}+H_{1}v_{3}+H_{3}v_{5}+...+H_{2n-1}v_{2n+1}$ is called as the axis of
LC helix.
\end{case}

\section{\textbf{Conclusion}}

Recently, a new type of helix, called slant helix, is studied $[3]$. In our
study, we define another type of helix in the space form, named $LC$ helix.
One of significance of this helix is that if we take the Euclidean space $%
E^{3}$ in the place of space form, the definition of $LC$ helix coincides
with the definition of helix in $E^{3}$.

\end{document}